\documentclass[11pt]{article}

\usepackage[T2A]{fontenc}
\usepackage[cp1251]{inputenc}
\usepackage{amsthm}%%%%%%%% theorem definition
\usepackage{amsfonts}
\usepackage{amssymb}
\usepackage{amsmath}
\usepackage{cite}%%%%% citations

\makeatletter
\renewcommand{\@biblabel}[1]{#1.\hfill}
\makeatother
 \usepackage{amsthm}
 \usepackage{bm}
 \makeatletter
\renewenvironment{proof}[1][\proofname]{%
  \par\pushQED{\qed}%
  \trivlist
  \item[\hskip\labelsep \textbf{\itshape #1}]%
}{%
  \popQED\endtrivlist\@endpefalse
}
\makeatother
\begin{document}
%%%%%%%%%%%%% begin theorem definition %%%%%%%%%%%%%%%%%%
\newtheoremstyle{mytheorem}
  {\topsep}   % ABOVESPACE
  {\topsep}   % BELOWSPACE
  {\itshape}  % BODYFONT
  {}       % INDENT (empty value is the same as 0pt)
  {\bfseries} % HEADFONT
  {  }         % HEADPUNCT
  {5pt plus 1pt minus 1pt} % HEADSPACE
   { }          % CUSTOM-HEAD-SPEC
   %{\thmname{#1}\thmnumber{ #2}\thmnote{ {\normalfont(#3)}}.}
\newtheoremstyle{myremark}
  {\topsep}   % ABOVESPACE
  {\topsep}   % BELOWSPACE
  {\upshape}  % BODYFONT
  {}       % INDENT (empty value is the same as 0pt)
  {\bfseries\itshape} % HEADFONT
  {  }         % HEADPUNCT
  {5pt plus 1pt minus 1pt} % HEADSPACE
  { }          % CUSTOM-HEAD-SPEC\cite{}
   %{\thmname{#1}\thmnumber{ #2}\thmnote{ {\normalfont(#3)}}.}
\theoremstyle{mytheorem}
\newtheorem{theorem}{Theorem}[section]
 \newtheorem{theorema}{Theorem}
 \newtheorem*{SD}{Skitovich--Darmois theorem}
 \newtheorem*{I}{Ibragimov theorem}
 \newtheorem*{R}{Ramachandran theorem}
\newtheorem{proposition}[theorem]{Proposition} 
\newtheorem{lemma}[theorem]{Lemma} 
\newtheorem{corollary}[theorem]{Corollary} 
\newtheorem{definition}[theorem]{Definition} 
\theoremstyle{myremark}
\newtheorem{remark}[theorem]{Remark} 
%%%%%%%%%%%%%%%%%%%%% end theorem definition %%%%%%%%%%%%%%%%%%
\noindent This article has been accepted  for publication 

\noindent in the journal  "Potential Analysis"

\bigskip

\noindent{\textbf{\Large A Group Analogue of  the Ghurye--Olkin--Ibragimov Theorem}}

\bigskip

\noindent{\textbf{\large Gennadiy Feldman}} 
(ORCID ID  https://orcid.org/0000-0001-5163-4079)

\bigskip

\noindent{\textbf{Abstract}}

\noindent According to the classical Skitovich--Darmois theorem, the Gaussian distribution
on the real line is characterized by the independence of two linear forms of a finite
number of independent random variables $\xi_j$. This theorem has been extended in various
directions. In particular, S.G.~Ghurye and I.~Olkin established an analogous result
for the case where $\xi_j$ are $n$-dimensional independent random vectors and the
coefficients of the linear forms are invertible $n\times n$ matrices.
Subsequently, A.A.~Zinger and later I.A.~Ibragimov investigated linear forms of 
an infinite
sequence of $n$-dimensional independent random vectors. In the present paper, 
we investigate, 
for the first time, linear forms of a 
sequence of independent random variables taking values in a second-countable 
locally compact Abelian group $X$ under either of the following conditions: 
$X$ contains no nontrivial compact subgroups, or $X$ contains no subgroups 
topologically isomorphic to the circle group and has a finite topological 
automorphism group. Furthermore, we investigate the case where the 
independent random variables take values in an ${\bm a}$-adic solenoid 
$\Sigma_{\bm a}$. In all these settings, the coefficients of the linear forms 
are topological automorphisms of the corresponding group.

\bigskip

\noindent{\bf Keywords}   Skitovich--Darmois theorem $\cdot$ Probability 
distribution $\cdot$ 
Locally compact Abelian group 

\bigskip

\noindent {\bf Mathematics Subject Classification (2020)}    
43A25 $\cdot$ 43A35 $\cdot$ 60B15 $\cdot$ 62E10

\section{Introduction}

The following theorem, independently proved by V.P.~Skitovich \cite{Ski}
and G.~Darmois \cite{Dar} (see also \cite[\S 3.1]{KaLiRa}),
is one of the most well-known characterization theorems in mathematical statistics.
\begin{SD}
Let $a_j, b_j$,  $j=1, 2, \dots, n$,  $n \ge 2$, be nonzero real numbers, and let 
$\xi_j$  be real-valued 
independent random variables.  
If the linear forms
$$
L_1 = \sum_{j=1}^n a_j \xi_j \   \text{and} \   L_2 = \sum_{j=1}^n b_j \xi_j
$$
are independent, then all $\xi_j$ are Gaussian.
\end{SD}
The Skitovich--Darmois theorem has been developed in many directions. We mention a few of them.  
S.G.~Ghurye and I.~Olkin \cite{GhurO} proved a similar theorem in the case when the $\xi_j$ take values in the space $\mathbb{R}^n$,  and $a_j$, $b_j$ 
are invertible $n\times n$ matrices.  
Yu.~V.~Linnik and A.A.~Zinger \cite{LZ}, and later G.P.~Chistyakov and F.~G\"{o}tze \cite{CG}, considered linear forms with random coefficients.  
F.~Lehner \cite{Le} showed that, in general, the Skitovich--Darmois theorem does not hold in the framework of free probability.  
A.M.~Kagan \cite{Ka} proved a theorem that can be considered a generalization of the Skitovich--Darmois theorem for dependent random variables and several linear forms.  
A.M.~Kagan and G.~Sz\'{e}kely \cite{KaS} studied the case of so-called $Q$-independent random variables.  
Based on a characterization of polynomials as solution sets of certain functional
equations, J.M.~Almira~\cite{Al1} proposed a new proof of the Ghurye--Olkin theorem.
In a large number of works by different authors, the case where 
$\xi_j$ take values in a locally compact Abelian group $X$ from a certain
class, and the coefficients $a_j$ and $b_j$ are topological automorphisms of $X$, 
has been studied (see, e.g., \cite{FeGra1, Fe10, Ma1, Fe16, M2020, FeGra2}, 
as well as \cite[Chapters~IV and~V]{FeBook} and \cite[Chapter~III]{F}, 
where further references can be found). In all the works mentioned above, 
linear forms of a finite number of independent random variables were considered.

The problem of generalizing the Skitovich--Darmois theorem to linear forms
of an infinite sequence of independent random variables was posed by 
Yu.V.~Linnik \cite{Lin}.
Using the approach he proposed, this problem was first investigated
by L.V.~Mamai \cite{M} and later by B.~Ramachandran \cite[Theorem~8.2.2]{R}.
The most significant result in this direction was obtained by I.A.~Ibragimov \cite{I1}.
An analogue of this result for the space $\mathbb{R}^n$, that is, a generalization
of the Ghurye--Olkin theorem to linear forms of an infinite sequence
of independent random vectors, was proved by A.A.~Zinger \cite{Z} and subsequently
strengthened by I.A.~Ibragimov \cite{I2}.
It should be noted that all the above-mentioned results were obtained under various restrictions
on the coefficients of the linear forms.

The main result of this paper is a generalization of Ibragimov's theorem \cite{I2}
to linear forms of an infinite sequence of independent random variables taking values
in a second-countable locally compact Abelian group $X$ with no nontrivial compact subgroups.
The coefficients of the linear forms are assumed to be topological automorphisms of $X$.
We also investigate this characterization problem for certain other locally compact
Abelian groups. To the best of our knowledge, the problem of characterizing probability
distributions by the independence of linear forms of an infinite sequence of independent
random variables has not previously been studied for groups.

In the paper, we use standard results from abstract harmonic analysis 
(see, e.g., \cite{HeRo1}).  
Let $X$ be a locally compact Abelian group, and let $Y$ denote its character 
group.  For $x \in X$, we denote by $(x,y)$ the value of a character $y \in Y$ at $x$.  
We write $X^2 = X \times X$ for the direct product 
of $X$ with itself.
Let $\operatorname{Aut}(X)$ be the group of all topological automorphisms of $X$, 
and let $I$ denote the identity automorphism of $X$.  
A closed subgroup $G$ of $X$ is called {characteristic} if $\alpha(G) = G$ 
for all $\alpha \in \operatorname{Aut}(X)$.  
Let
 $\alpha\in\operatorname{Aut}(X)$. 
 The adjoint automorphism $\widetilde\alpha\in\operatorname{Aut}(Y)$
is defined by the formula $(\alpha x,
y)=(x, \widetilde\alpha y)$ for all $x\in X$, $y\in
Y$.  
Let $K$ be a closed 
subgroup of $X$. Denote by $A(Y, K) = \{y \in Y: (x, y) =1$ for all $x\in
K\}$ the annihilator of $K$. 

Let $\mu$ and $\nu$ be probability 
distributions on the group $X$. The convolution
$\mu*\nu$ is defined by the formula
$$
\mu*\nu(B)=\int\limits_{X}\mu(B-x)d \nu(x)
$$
for any Borel subset $B$ of $X$.
 Denote by
$$
\widehat\mu(y) =
\int_{X}(x, y)d \mu(x), \quad y\in Y,$$   the characteristic function 
(Fourier transform) of $\mu$.  
Define the distribution $\bar \mu $ by the formula
 $\bar \mu(B) = \mu(-B)$ for any Borel  subset $B$ of $X$.
We have $\widehat{\bar{\mu}}(y)=\overline{\widehat\mu(y)}$.  
For $x\in X$, denote by $E_x$ the Dirac mass at $x$, i.e., 
the degenerate distribution
 with support at $x$. 
 If $X$ is a compact group, denote by $m_X$ the normed Haar
 measure on $X$. 
 
 Let $X$ be a second-countable locally compact Abelian group.
 For a random variable  $\xi$ 
 taking  values in 
 $X$, denote by $\mu_\xi$ its distribution.  Let
$\xi_j$, $j = 1, 2, \dots$, be  independent random variables 
taking values in $X$.
We say that the series $\sum\limits_{j=1}^\infty \xi_j$ converges in distribution if 
there exists a random variable $L$ taking values in $X$ 
such that the sequence of partial sums 
$S_n=\sum\limits_{j=1}^n\xi_j$ 
converges in distribution\footnote{On the real line, 
convergence almost surely, in probability, and in distribution are 
equivalent for series of independent random variables.} to $L$ as $n\to \infty$, i.e.,
the distributions $\mu_{S_n}$ converge weakly to $\mu_{L}$. 
In this case, we write 
  $L=\sum\limits_{j=1}^\infty\xi_j$.
This is equivalent to saying that  the sequence of partial products
$P_n(y)=\prod\limits_{j=1}^n\widehat\mu_{\xi_j}(y)$ on 
compact subsets of the character group $Y$ converges uniformly 
to $\widehat\mu_L(y)$ as $n\to\infty$,  and we then write
  $\widehat\mu_L(y)=\prod\limits_{j=1}^\infty\widehat\mu_{\xi_j}(y)$.
We note that unlike classical infinite products in analysis, the limit of partial products of characteristic functions may vanish at some points, which does not affect convergence in distribution.

\section{Lemmas}

To prove the main theorem, we require some lemmas. 
 
\begin{lemma}\label{l1}
Let $X$ be a second-countable locally compact Abelian group with character group $Y$, 
and let $\alpha_j, \beta_j$, $j=1,2,\dots$, be topological automorphisms of $X$. 
Let $\xi_j$ be independent random variables taking values in $X$  with distributions $\mu_j$. 
Assume that
$$L_1=\sum\limits_{j=1}^\infty\alpha_j\xi_j\   \text{and} \  
 L_2=\sum\limits_{j=1}^\infty\beta_j\xi_j.$$
Then the linear forms $L_1$ and $L_2$ are independent
if and only if the characteristic functions $\widehat\mu_j(y)$ satisfy the equation
\begin{equation}
\label{e2}\prod_{j=1}^\infty\widehat\mu_j(\widetilde\alpha_j u+\widetilde\beta_j
v)=\prod_{j=1}^\infty\widehat\mu_j(\widetilde\alpha_j
u)\prod_{j=1}^\infty\widehat\mu_j(\widetilde\beta_j v), \quad u, v \in
Y, 
\end{equation}
where the infinite products in the both sides of equation~$(\ref{e2})$
 converge uniformly on each 
 compact subset of the group $Y^2$.
\end{lemma}
\begin{proof} The proof of the lemma is carried out 
as in the case of real-valued random 
variables. Let $\xi$ be a
random variable  taking values in $X$ and distribution 
$\mu$, and let $\alpha
\in \operatorname{Aut}(X)$.  The
characteristic function of the random variable  
$\alpha \xi$ is
equal to
 $\widehat \mu(\widetilde \alpha y)$.
We   note that the linear forms  $L_1$ and $L_2$
are independent if and only if the  equality $\mathbf{E}[(L_1,u)(L_2,v)]=\mathbf{E}[(L_1,u)]\mathbf{E}[(L_2,v)]$, i.e.,
\begin{equation}\label{eq1} 
 \mathbf{E}\left[\left(\sum_{j=1}^\infty\alpha_j\xi_j, u\right)
 \left(\sum_{j=1}^\infty\beta_j\xi_j,
v\right)\right]=\mathbf{ E}\left[\left(\sum_{j=1}^\infty\alpha_j\xi_j,
u\right)\right]\mathbf{ E}\left[\left(\sum_{j=1}^\infty\beta_j\xi_j, v\right)\right]
\end{equation}
holds for all  $u, v \in Y$. 
Since the series 
$\sum\limits_{j=1}^\infty\alpha_j\xi_j$ and
$\sum\limits_{j=1}^\infty\beta_j\xi_j$ 
 converge in distribution to $L_1$ and $L_2$,   
 the infinite products 
 $ 
 \prod\limits_{j=1}^\infty\widehat\mu_j(\widetilde\alpha_j
y)$ and 
$\prod\limits_{j=1}^\infty\widehat\mu_j(\widetilde\beta_j y)
$  converge uniformly on each 
 compact subset of the group $Y$ to the characteristic functions   
 $\widehat\mu_{L_1}(y)$ and  $\widehat\mu_{L_2}(y)$ respectively, and also 
 the infinite product
 $\prod\limits_{j=1}^\infty\widehat\mu_j(\widetilde\alpha_j u+\widetilde\beta_j
v)$ converges uniformly on each 
 compact subset of the group $Y^2$ to the characteristic function    
 $\widehat\mu_{(L_1, L_2)}(u, v)$. 
 
 Taking into account that the random variables $\xi_j$ are independent 
and that $\widehat\mu_j(y)=\mathbf{E}[(\xi_j, y)]$, and using the definition 
of convergence in distribution, we transform the left-hand side of equality~(\ref{e2}) 
as follows:
\begin{multline*}
\mathbf{E}\left[\left(\sum_{j=1}^\infty\alpha_j\xi_j, u\right)
 \left(\sum_{j=1}^\infty\beta_j\xi_j, v\right)\right] 
= \lim_{n \to \infty} \mathbf{E}\left[\left(\sum_{j=1}^n\alpha_j\xi_j, u\right)
 \left(\sum_{j=1}^n\beta_j\xi_j, v\right)\right] \\
= \lim_{n \to \infty} \mathbf{E}\left[\prod_{j=1}^n (\xi_j, \widetilde\alpha_j u + \widetilde\beta_j v)\right] 
= \lim_{n \to \infty} \prod_{j=1}^n \mathbf{E}[(\xi_j, \widetilde\alpha_j u + \widetilde\beta_j v)] \\
= \prod_{j=1}^\infty \widehat\mu_j(\widetilde\alpha_j u + \widetilde\beta_j v), \quad u, v \in Y.
\end{multline*}
Analogously, we transform the right-hand side of equality~(\ref{e2}) as follows:
\begin{multline*}
\mathbf{E}\left[\left(\sum_{j=1}^\infty\alpha_j\xi_j, u\right)\right]\mathbf{E}\left[\left(\sum_{j=1}^\infty\beta_j\xi_j, v\right)\right]\\
= \lim_{n \to \infty} \mathbf{E}\left[\left(\sum_{j=1}^n\alpha_j\xi_j, u\right)\right] 
\lim_{n \to \infty} \mathbf{E}\left[\left(\sum_{j=1}^n\beta_j\xi_j, v\right)\right] \\
= \lim_{n \to \infty} \mathbf{E}\left[\prod_{j=1}^n (\xi_j, \widetilde\alpha_j u)\right] 
\lim_{n \to \infty} \mathbf{E}\left[\prod_{j=1}^n (\xi_j, \widetilde\beta_j v)\right] \\
= \lim_{n \to \infty} \prod_{j=1}^n \mathbf{E}[(\xi_j, \widetilde\alpha_j u)] \lim_{n \to \infty} \prod_{j=1}^n \mathbf{E}[(\xi_j, \widetilde\beta_j v)] \\
= \prod_{j=1}^\infty \widehat\mu_j(\widetilde\alpha_j u) \prod_{j=1}^\infty \widehat\mu_j(\widetilde\beta_j v), \quad u, v \in Y.
\end{multline*}
As a result we obtain that the characteristic functions $\widehat\mu_j(y)$ 
satisfy   equation~(\ref{e2}).
\end{proof}

\begin{lemma}\label{th1}
Let $X$ be a countable discrete Abelian group with character group $Y$,  
and   
let $\alpha_j, \beta_j$, $j=1, 2, \dots$,
be automorphisms  of   $X$.
Let $\xi_j$ be independent random variables taking 
values in $X$  with distributions $\mu_j$.   
Assume that the following conditions hold:    
    \renewcommand{\labelenumi}{\rm(\roman{enumi})}
\begin{enumerate}  
\item	

$L_1=\sum\limits_{j=1}^\infty\alpha_j\xi_j$ and 
$L_2=\sum\limits_{j=1}^\infty\beta_j\xi_j$;

\item

the linear forms $L_1$ and $L_2$ are independent;

\item	

the characteristic functions $\widehat\mu_{L_1}(y)$ and $\widehat\mu_{L_2}(y)$ do not vanish.

 \end{enumerate}
Then all $\mu_j$ are degenerate distributions.
\end{lemma}

\begin{proof}  Since $X$ is a discrete Abelian group, $Y$ is a compact Abelian group. 
By Lemma~\ref{l1}, it follows from
conditions (i) and (ii) that the characteristic functions  $\widehat \mu_j(y)$ 
satisfy equation~(\ref{e2}) and 
the infinite products in the both sides of equation~(\ref{e2})
 converge uniformly on the group $Y^2$.
It follows from condition (iii) that 
all characteristic functions $\widehat \mu_j(y)$ do not vanish on 
 $Y$. Set $\nu_j = \mu_j * \bar \mu_j$. Then
 $\widehat \nu_j(y) = |\widehat \mu_j(y)|^2 > 0$   for all $y \in Y$.
Obviously, the characteristic functions  $\widehat \nu_j(y)$ also
satisfy equation~(\ref{e2}). 
Let $f_j(y) = - \ln{\widehat \nu_j(y)}$, $y\in Y$. 
By taking logarithms on both sides of equation~(\ref{e2}) for 
the characteristic functions $\widehat \nu_j(y)$, we obtain
 that the functions $f_j(y)$ satisfy the
equation
\begin{equation}
\label{e3} \sum_{j = 1}^{\infty}  f_j(\widetilde\alpha_j u
+\widetilde\beta_j v) = \sum_{j = 1}^{\infty}  f_j(\widetilde\alpha_j u
)+\sum_{j = 1}^{\infty}  f_j(\widetilde\beta_j v),\quad u, v \in Y,
\end{equation}
where the series in the  both sides of equation~(\ref{e3}) converge uniformly on 
 the group $Y^2$. Integrating  equation~(\ref{e3}) over the group $Y$ with
respect to the Haar measure $m_Y$ in the variable $u$  and 
using that  the Haar
 measure $m_Y$ is $Y$-invariant, we find
\begin{equation}
\label{e4} \sum_{j = 1}^{\infty} f_j(\widetilde \beta_j
v)= 0, \quad v\in Y.
\end{equation}
Taking into account that $f_j(y)\le 0$, it follows from~(\ref{e4})  that   
$f_j(\widetilde \beta_j y)=0$. 
Since $\widetilde \beta_j\in\operatorname{Aut}(Y)$, we have 
$f_j(y)=0$
for all $y\in Y$, $j=1, 2,\dots$. This implies that
$\widehat\nu_j(y)= 1$ for all $y\in Y$. Hence  $\nu_j=E_0$, $j=1, 2,\dots$.  
It follows from this that all $\mu_j$ 
are degenerate distributions. 
\end{proof}

Note that, unlike the case where independent random variables take values on the real
line or in the space $\mathbb{R}^n$, in Lemma~\ref{th1} we 
impose no restrictions on the
coefficients $\alpha_j$ and $\beta_j$. At the same time, we assume that the characteristic
functions $\widehat\mu_{L_1}(y)$ and $\widehat\mu_{L_2}(y)$ do not vanish.
 
\begin{remark}\label{r2}
Let $X$ be a countable discrete Abelian group, and let 
$\xi_j$, $j=1, 2,\dots$, be independent random variables taking values in $X$  
with degenerate distributions $E_{x_j}$, $x_j\in X$. 
If the series $\sum\limits_{j=1}^\infty\xi_j$ converges in distribution, 
then only a finite number of elements $x_j$ differ  from zero. 
Taking this into account, we see that if conditions (i)--(iii) of 
Lemma~\ref{th1} are satisfied, then in fact  the series 
  $$L_1=\sum\limits_{j=1}^\infty\alpha_j\xi_j\  \text{and} \  
  L_2=\sum\limits_{j=1}^\infty\beta_j\xi_j$$ in Lemma~\ref{th1} are finite sums.
\end{remark}

\begin{remark}\label{nnr2}
Generally speaking, Lemma~\ref{th1}  fails if condition (iii) is not satisfied. 
We construct an example for a finite Abelian group of odd order. 

Let $X$ 
be a compact Abelian group with character group $Y$. We recall that 
the characteristic function $\widehat m_X(y)$ is of the form
\begin{equation}
\label{fe22.1}\widehat m_X(y)=
\begin{cases}
1 & \text{\ if\ }\ \ y=0,
\\ 0 & \text{\ if\ }\ \ y\ne 0.
\end{cases}
\end{equation} 

Assume that $X$ is a finite Abelian group of odd order.
Let $\alpha_j, \beta_j$, $j = 1, 2, \dots$, be automorphisms of $X$ such that
$\alpha_1 = \alpha_2 = \beta_1 = I$, $\beta_2 = -I$, and all remaining
$\alpha_j$ and $\beta_j$ are arbitrary automorphisms of $X$.
Let $\xi_j$ be independent random variables taking values in $X$ with
distributions $\mu_j$ such that $\mu_1 = \mu_2 = m_X$.
Then $\mu_{L_1} = \mu_{L_2} = m_X$, and condition~(i) is satisfied.  

We now verify that condition~(ii) is also fulfilled. By Lemma~\ref{l1},
it suffices to check that the characteristic functions $\widehat{\mu}_j(y)$
satisfy equation~(\ref{e2}), which takes the form
\begin{multline}
\label{nne2}\widehat m_X(u+v)\widehat m_X(u-v)\prod_{j=3}^\infty\widehat\mu_j(\widetilde\alpha_j 
u+\widetilde\beta_j
v)\\=\widehat m_X^2(u)\widehat m_X^2(v)\prod_{j=3}^\infty\widehat\mu_j(\widetilde\alpha_j
u)\prod_{j=3}^\infty\widehat\mu_j(\widetilde\beta_j v), \quad u, v \in
Y. 
\end{multline}
In view of~(\ref{fe22.1}), the right-hand side of equation~(\ref{nne2}) 
is equal to zero if  $u\ne 0$. Suppose that $u\ne 0$
and  the left-hand side of equation~(\ref{nne2}) is not equal to zero. 
In view of~(\ref{fe22.1}), it follows from this that $u+v=0$ and $u-v=0$.
Hence $2u=0$. Since $X$ is a finite Abelian group of odd order and 
the group $Y$ is isomorphic to $X$, the group $Y$ contains
no elements of order 2. This implies that $u=0$. 
The obtained contradiction shows that the left-hand side of equation~(\ref{nne2}) 
is equal to zero. Thus, we see that equation~(\ref{nne2}) 
 becomes an equality if $u\ne 0$. It is obvious that equation~(\ref{nne2}) 
 becomes an equality if $u=0$. So, the characteristic functions 
 $\widehat\mu_j(y)$ satisfy equation~(\ref{nne2}) and hence condition 
 (ii) is satisfied.
\end{remark}

It is convenient to formulate the following well-known statement as a lemma 
(for the proof see, e.g., \cite[Proposition 2.10]{F}).

\begin{lemma}\label{l4}
Let  $X$ be a locally
	compact Abelian group with character group $Y$,  and let 
	$H$ be a closed
	subgroup of $Y$.
	Let $\mu$  be a distribution on $X$  such that
	$\widehat\mu(y)=1$  for all $y\in H$. Then $\mu$ is supported in 
	the annihilator $A(X,H)$.
	\end{lemma}

Let us formulate now Ibragimov's theorem. Denote by $\|\cdot\|$ 
a norm on the 
space of  real $n\times n$~matrices.

\begin{I}{\rm(\!\!\cite{I2})}
Let $A_j$ and  $B_j$, $j=1, 2, \dots$, be real invertible $n\times n$~matrices, 
and let $\xi_j$ be 
independent random vectors  taking values in the space $\mathbb{R}^n$ 
with distributions 
$\mu_j$.  
Assume that the following conditions hold:
\renewcommand{\labelenumi}{\rm(\roman{enumi})}
\begin{enumerate}
\item	
$L_1=\sum\limits_{j=1}^\infty A_j\xi_j$ and 
$L_2=\sum\limits_{j=1}^\infty B_j\xi_j$;
\item
the linear forms $L_1$ and $L_2$ are independent;

\item

at least one of the sequences $\{A_jB_j^{-1}\}$ or 
$\{B_jA_j^{-1}\}$ is bounded;
\item

 the sequence
$\{\|A_jB_j^{-1}\|\cdot\|B_jA_j^{-1}\|\}$ is bounded.
\end{enumerate}
Then all $\mu_j$ are Gaussian distributions.
\end{I}

\section{Main theorem}

Before proceeding to the formulation and proof of the main theorem, 
we note the following. 
Let $X$ be a locally 
compact Abelian group with character group $Y$, and let $K$ be a 
closed subgroup of $X$.  
The character group of the factor group $X/K$ is topologically isomorphic to the annihilator $A(Y,K)$.  
Let $\alpha \in \operatorname{Aut}(X)$ and assume that $\alpha(K) = K$.
Then the restriction of $\alpha$ to $K$ is a topological automorphism of $K$,
which we denote by $\alpha|_K$.
The automorphism $\alpha$ induces a topological automorphism
$\widehat{\alpha}$ of the factor group $X/K$, defined by
$\widehat{\alpha}[x] = [\alpha x]$, $[x] \in X/K$.
Doing so, the restriction of $\widetilde{\alpha}$ to 
the annihilator $A(Y,K)$ 
coincides with the adjoint automorphism of $\widehat{\alpha}$.

Assume  $X=\mathbb{R}^n\times G$,
where $G$ is a discrete  Abelian group. Then $\mathbb{R}^n$ is 
the connected component of zero of $X$ and hence 
 a characteristic subgroup. The factor group
$X/\mathbb{R}^n$ is topologically isomorphic to $G$ and we can consider 
the induced automorphism $\widehat\alpha$ of $X/\mathbb{R}^n$
as an automorphism of the group $G$.
For  $\alpha\in\textrm{Aut}(X)$ we identify the restriction 
$A=\alpha|_{\mathbb{R}^n}$
with the corresponding invertible  $n\times n$ matrix. 

Now we can formulate the main theorem. It follows easily from the structure 
theorem for locally compact Abelian groups (\!\!\cite[(24.30)]{HeRo1}) 
that a second-countable 
locally compact Abelian group $X$ contains no 
 nontrivial compact subgroups
 if and only if  $X$ is topologically isomorphic 
to a group of the form $\mathbb{R}^n\times G$,
where $G$ is a countable discrete  torsion free Abelian group.
	
\begin{theorem}\label{th4}
Let  $X=\mathbb{R}^n\times G$,
where $G$ is a countable discrete  torsion free Abelian group. Let  
$\alpha_j$, $\beta_j$, $j=1,2, \dots$,
be topological automorphisms  of  $X$ such that the matrices 
$A_j= \alpha_j|_{\mathbb{R}^n}$ and  $B_j=\beta_j|_{\mathbb{R}^n}$ satisfy
conditions $\rm(iii)$ and $\rm(iv)$ of Ibragimov's theorem.
Let $\xi_j$ be independent random variables taking
values in the group $X$  with distributions $\mu_j$.   Assume that the following 
conditions hold:
\renewcommand{\labelenumi}{\rm(\roman{enumi})}
\begin{enumerate}

\item

$L_1=\sum\limits_{j=1}^\infty\alpha_j\xi_j$ and $L_2=
\sum\limits_{j=1}^\infty\beta_j\xi_j$;

\item

the linear forms $L_1$ and $L_2$ are independent.

\end{enumerate}
Then all $\mu_j$ are shifts of Gaussian distributions in $\mathbb{R}^n$.
\end{theorem}

\begin{proof}
Considering the new independent random variables $\kappa_j = \alpha_j \xi_j$, 
we may assume, without loss of generality, that $\alpha_j = I$, and hence
$A_j = I$,
$j = 1, 2, \dots$.
Denote by $Y$ and $H$ the character groups of the groups $X$ and $G$, 
respectively. Since $G$ is a discrete torsion-free Abelian group, 
$H$ is a compact connected Abelian group  (\!\!\cite[(23.17), (24.19)]{HeRo1}). 
The group $Y$ is 
topologically isomorphic to the group $\mathbb{R}^n \times H$. 
To avoid introducing new notation, we assume that 
$Y = \mathbb{R}^n \times H$.
By Lemma~\ref{l1}, the characteristic functions  $\widehat \mu_j(y)$ 
satisfy equation~(\ref{e2}) which takes the form
\begin{equation}
\label{e5}\prod_{j=1}^\infty\widehat\mu_j(u+\widetilde\beta_j
v)=\prod_{j=1}^\infty\widehat\mu_j(u)\prod_{j=1}^\infty\widehat\mu_j(\widetilde\beta_j v), 
\quad u, v \in
Y,
\end{equation}
where the infinite products in both sides of equation~(\ref{e5})
converge uniformly on the group $Y^2$.

1. First, we prove the theorem under the assumption that
 $\widehat\mu_j(y)\ge 0$ for all $y\in Y$, $j=1, 2, \dots$. 

We recall that an element of a locally compact Abelian group is 
called a compact element if the cyclic subgroup generated by it has compact closure.
Since $H$ consists of all compact elements of the group $Y$, 
the subgroup $H$ is a characteristic subgroup of $Y$, and we can 
consider the restriction of equation~(\ref{e5}) to $H$.  Let
\begin{equation*}
A=\Biggl\{ h\in H:\ \text{for every }\varepsilon>0 \text{ there exists } 
n_\varepsilon(h) \text{ such that }
1-\prod_{j=n_\varepsilon(h)}^\infty \widehat{\mu}_j(h) < \varepsilon \Biggr\}.
\end{equation*}
The set $A$ is non-empty since $0 \in A$. Furthermore, the continuity of the 
functions $\prod\limits_{j=n}^\infty \widehat{\mu}_j(y)$ for each positive integer $n$ 
implies that $A$ is open.
We note that any characteristic function $\widehat{\mu}(h)$ on a locally compact 
Abelian group satisfies the inequality  
\begin{equation}\label{09.02.1}
1-\operatorname{Re}\ \widehat\mu (u+v)\le 2[(1-\operatorname{Re}\ \widehat\mu
(u))+(1-\operatorname{Re}\ \widehat\mu( v))],
\quad u, v \in H
\end{equation}
(see, e.g., \cite[Chapter~IV, pp.~82, 88]{Parthasarathy}).
It follows from~(\ref{09.02.1}) that $A$ is a subgroup of $H$, and hence $A$ 
 is an open subgroup of $H$.
As $H$ is a connected group, we conclude that $A = H$.

Take an element
 $h\in H$. Then $h\in A$ and hence there exists a natural number 
 $n_h$ such that
 the infinite product 
 $\prod\limits_{j=n_h}^\infty\widehat\mu_j(h)$ converges to a positive number. This implies
 that there exists a neighbourhood $U_h$ of $h$ such that for any $u\in U_h$
 the infinite product  
 $\prod\limits_{j=n_h}^\infty\widehat\mu_j(u)$ also 
 converges to a positive number, 
 depending on $u$.   The family of neighbourhoods $U_h$, where $h$ runs 
 through $H$, 
  forms an open covering of the group $H$. Since $H$ is a compact, let 
  $\left\{U_{h_i}\right\}$ be a finite subcovering. 
Set $l=\max\limits_{i}\{n_{h_i}\}$. Then
\begin{equation}\label{26.12.1}
\prod\limits_{j=l+1}^\infty\widehat\mu_j(h)>0 \quad \mbox{for all} \ \ h\in H.
\end{equation}
Arguing similarly we find $m$ such that
\begin{equation}\label{26.12.3}
\prod\limits_{j=m+1}^\infty\widehat\mu_j(\widetilde\beta_{j}h)>0 
\quad \mbox{for all} \ \ h\in H.
\end{equation}
Let  
$$p=\max\{l, m\},\quad {\cal A}_j=\{h\in H:\widehat\mu_{j}(h)> 0\},\quad
{\cal A}=\bigcap\limits_{j=1}^\infty {\cal A}_j, $$$$ 
{\cal B}_j=\{h\in H:\widehat\mu_{j}(\widetilde\beta_{j} h)> 0\},\quad
{\cal B}=\bigcap\limits_{j=1}^\infty {\cal B}_j, \quad 
{\cal C}=\bigcap\limits_{j=1}^{p} \widetilde\beta_{j}({\cal B}).$$ 
It follows from~(\ref{26.12.1}) that  
\begin{equation}
\label{26.12.2}{\cal A}_j=H \quad \mbox{for all} \ \ j\ge p+1.
\end{equation}

Let us prove the inclusion 
\begin{equation}
\label{e6}{\cal A}+{\cal C}\subset {\cal A}.
\end{equation}
Take $u\in {\cal A}$, $w\in
{\cal C}$. Then $w=\widetilde\beta_{j} w_j$, where $w_j\in {\cal B}$, $j=1, 2,
\dots, p$. Substitute $u\in {\cal A}$, $v=w_1$ in equation~(\ref{e5}). 
We obtain
\begin{equation}
\label{ne7}\prod_{j=1}^\infty\widehat\mu_j(u+\widetilde\beta_j
w_1)=\prod_{j=1}^\infty\widehat\mu_j(u)\prod_{j=1}^\infty\widehat\mu_j(\widetilde\beta_j w_1), 
\quad u  \in
{\cal A}.\end{equation}
Since $u\in {\cal A}$ and $w_1\in {\cal B}$, we have $\widehat\mu_j(u)>0$ and
 $\widehat\mu_j(\widetilde\beta_j w_1)>0$, $j=1, 2, \dots$. In view 
 of~(\ref{26.12.1}) and~(\ref{26.12.3}), the right-hand side of   
equality~(\ref{ne7}) is different from zero.
Hence the left-hand side  of   
equality~(\ref{ne7}) is also different from zero. In particular,
$\widehat\mu_{1}(u+\widetilde\beta_{1}
w_1)=\widehat\mu_{1}(u+w)>0$,
 i.e.,
$u+w\in {\cal A}_1$. Arguing similarly we get that $u+w\in {\cal A}_k$ for $k=2, 3,
\dots, p$. Taking into account~(\ref{26.12.2}), this implies 
that $u+w\in {\cal A}$ and~(\ref{e6}) is proved.

For a natural number $k$, let 
$$(k){\cal C}=\{h\in H: h=h_1+\dots+h_k, \ h_j\in {\cal C}\}.$$ 
We obtain from~(\ref{e6}) that 
 \begin{equation}
\label{e9}{\cal A}+\bigcup\limits_{k=1}^\infty (k){\cal C}\subset {\cal A}.
\end{equation}

 It follows from~(\ref{26.12.2}) that $B_j = H$ for all $j \ge p+1$, 
so that the set $\mathcal{B} = \prod\limits_{j=1}^\infty B_j$ and, hence, 
$\mathcal{C}$ are open sets.
  This implies that the set
 $\bigcup\limits_{k=1}^\infty (k){\cal C}$ is an open subgroup of $H$.
As the group $H$ is connected, this yields
 \begin{equation}
\label{e10}H=\bigcup\limits_{k=1}^\infty (k){\cal C}.
\end{equation}
It follows from~(\ref{e9}) and~(\ref{e10}) that ${\cal A}=H$,
i.e.,   $\widehat\mu_{j}(h)> 0$ for all
$h\in H$, $j=1, 2, \dots$. Since
$$
\widehat\mu_{L_1}(y)=\prod_{j=1}^\infty\widehat\mu_{j}(y), 
\quad \widehat\mu_{L_2}(y)=\prod_{j=1}^\infty\widehat\mu_{j}(\widetilde\beta_{j} y), 
\quad y\in
Y,
$$
taking into account~(\ref{26.12.1}) and~(\ref{26.12.3}), we get that the restrictions of
 the characteristic functions 
$\widehat\mu_{L_1}(y)$ and $\widehat\mu_{L_2}(y)$ to the subgroup $H$  do not
vanish. 

We note that the restriction of the characteristic function 
$\widehat{\mu}_j(y)$ to $H$ coincides with the characteristic function of a 
distribution $\lambda_j$ on the group $G$, which is determined by the formula
$\lambda_j(B)=\mu_j(\mathbb{R}^n\times B)$ for any subset $B$ of $G$. 
Since $H = A(Y, \mathbb{R}^n)$, the restrictions of 
$\widetilde{\beta}_j$ to $H$ coincide with the adjoint automorphisms of 
$\widehat{\beta}_j$, where $\widehat{\beta}_j$ are the automorphisms induced on 
the factor group $X/\mathbb{R}^n$, which is topologically isomorphic to $G$.
 Let $\eta_j$ be the independent random variables
taking values in $G$, with distributions $\lambda_j$. 
Taking into account equation~(\ref{e5}), it follows from Lemma~\ref{l1} that
 the linear forms
\[
M_1 = \sum_{j=1}^n \eta_j \ \text{and} \  M_2 = \sum_{j=1}^n \widehat\beta_j \eta_j
\]
are independent. 
The characteristic functions 
$\widehat{\mu}_{M_1}(h)$ and $\widehat{\mu}_{M_2}(h)$ coincide with the restriction of
the characteristic functions 
$\widehat{\mu}_{L_1}(y)$ and $\widehat{\mu}_{L_2}(y)$ to the subgroup $H$.
Hence, they do not vanish. 
By Lemma~\ref{th1}, applied to the group $G$, we obtain that all $\lambda_j$ are
degenerate distributions. This implies that there exist
elements $g_j\in G$ such that $\widehat\lambda_j(h)=(g_j, h)$, $h\in H$, $j=1, 2, \dots$.
Taking into account that 
$\widehat\lambda_j(h)\ge 0$ for all $h \in H$, $j = 1, 2, \dots$, we conclude that
$g_1=g_2=\cdots=0$. Hence $\widehat\mu_j(h)=\widehat\lambda_j(h)=1$, $h \in H$, $j = 1, 2, \dots$,
and by Lemma~\ref{l4},
each of the distributions $\mu_j$ is supported in the annihilator 
$A(X, H)=\mathbb{R}^n$.
It means that we can consider $\xi_j$ as independent random vectors taking values 
in the space   $\mathbb{R}^n$, with distributions $\mu_j$. 
The linear forms $L_1$ and $L_2$ can be written as	
 $$L_1=\sum\limits_{j=1}^\infty \xi_j\ \text{and}\ 
 L_2=\sum\limits_{j=1}^\infty B_j\xi_j,
 $$
 and all
conditions  of Ibragimov's theorem for these linear forms hold. 
Applying Ibragimov's theorem to the independent 
random 
vectors $\xi_j$ and the linear forms $L_1$ and $L_2$, we obtain
 that all $\mu_j$ are Gaussian 
distributions in $\mathbb{R}^n$. Thus, we proved the theorem assuming that
all $\widehat\mu_j(y)\ge 0$, $y\in Y$.

2. We now prove the theorem in the general case. Set $\nu_j = \mu_j * \bar \mu_j$. Then
$\widehat \nu_j(y) = |\widehat \mu_j(y)|^2 \ge 0$ for all $y \in Y$.
The characteristic functions $\widehat \nu_j(y)$ also satisfy equation~(\ref{e5}). 
Denote by $\zeta_j$ independent random variables taking values in the group $X$ 
with distributions $\nu_j$. The series $\sum\limits_{j=1}^\infty\zeta_j$ and 
$\sum\limits_{j=1}^\infty\beta_j\zeta_j$ converge in distribution.
Let 
$$N_1=\sum\limits_{j=1}^\infty\zeta_j\  \text{and} \  
N_2=\sum\limits_{j=1}^\infty\beta_j\zeta_j.$$ By Lemma~\ref{l1}, the linear forms $N_1$ and $N_2$ are independent. 
Since $\widehat{\nu}_j(y) \ge 0$ for all $y \in Y$, as shown in Step~1,
all $\nu_j$ are Gaussian distributions in $\mathbb{R}^n$.
In view of $\nu_j = \mu_j * \bar \mu_j$, it is easy to verify that the distributions 
$\mu_j$  can be replaced by their shifts $\mu_j^\prime = \mu_j * E_{x_j}$, 
where $x_j \in X$, 
in such a way that the distributions $\mu_j^\prime$ are supported in $\mathbb{R}^n$.
As $\nu_j =\mu_j^\prime* \bar \mu_j^\prime$,
applying Cram\'er's theorem on the decomposition of the Gaussian distribution in 
$\mathbb{R}^n$, we conclude that all $\mu_j^\prime$ are Gaussian distributions 
in $\mathbb{R}^n$. Hence 
all $\mu_j$ are 
shifts of Gaussian distributions in $\mathbb{R}^n$.
\end{proof}

\section{Characterization of probability distributions on some other 
 groups by the independence of linear forms of an 
infinite sequence of\\ independent random variables}

Let ${\bm a} = (a_0, a_1, a_2, \dots)$, where $a_j$ are natural numbers 
with $a_j > 1$ for all $j$, 
and let $\Delta_{\bm a}$ denote the group of ${\bm a}$-adic integers 
(\!\!\cite[(10.2)]{HeRo1}). 
If $p$ is a prime number and ${\bm a} = (p, p, p, \dots)$, we write $\Delta_p$ 
for the corresponding group of ${\bm a}$-adic integers.
The group $\Delta_p$ is a ring (\!\!\cite[(10.10)]{HeRo1}).
As a
set $\Delta_{\bm a}$ coincides
with the Cartesian product $\mathop{\mbox{\rm\bf
P}}\limits_{n=0}^\infty\{0,1,\dots ,a_n-1\}$.

Consider the group
$\mathbb{R}\times\Delta_{\bm a}$. Let $u=(1, 0,\dots,0,\dots)\in \Delta_{\bm a}$, 
and let
 ${\cal B}$ be a subgroup of $\mathbb{R}\times\Delta_{\bm a}$ of the form
${\cal B}=\{(n,n{u})\}_{n=-\infty}^{\infty}$. The factor group 
$(\mathbb{R} \times \Delta_{\bm a})/{\cal B}$ is called an
${\bm a}$-adic solenoid and is denoted by $\Sigma_{\bm a}$. The group
$\Sigma_{\bm a}$ is compact, connected, and has dimension 1 
(\!\!\cite[(10.12), (10.13),
(24.28)]{HeRo1}). The character group of the group
$\Sigma_{\bm a}$ is topologically isomorphic to a discrete additive
group $
 H_{\bm a}$ of the rational numbers of the form
 \begin{equation*} 
H_{\bm a}=
\left\{\frac{m}{a_0a_1 \cdots a_n} : \ n = 0, 1,\dots; \ m
\in {\mathbb{Z}} \right\}
\end{equation*}
(\!\!\cite[(25.3)]{HeRo1}). It is convenient to consider
$H_{\bm a}$ as the character group of the group $\Sigma_{\bm a}$. 
We note that if $X$ is a locally compact connected Abelian group 
of dimension~1, then $X$ is topologically isomorphic to either
the real line, the circle group $\mathbb{T}$, or an
${\bm a}$-adic solenoid $\Sigma_{\bm a}$.

Denote by ${\mathbb Z}(m)$  the multiplicative group
of  $m$th roots of unity. 
Let $p$ be a prime number. Denote by ${\mathbb Z}(p^\infty)$ the 
multiplicative group
of  $p^n$th roots of unity, where $n$ goes through the nonnegative
integers.
We consider the group ${\mathbb Z}(p^\infty)$ equipped with the discrete 
topology. 
The groups ${\mathbb Z}(p^k)$, 
$k=0, 1, 2, \dots$,
are subgroups of ${\mathbb Z}(p^\infty)$. 
The character groups of ${\mathbb Z}(p^\infty)$ and $\Delta_p$ are topologically
isomorphic to $\Delta_p$ and ${\mathbb Z}(p^\infty)$, respectively.

Let us prove the following characterisation theorem
 for independent linear forms of an infinite number of independent 
 random variables.

\begin{theorem}\label{th5}
Let  
\begin{equation}
\label{ne10}
X = \mathbb{Z}(2^{m_1}) \times \mathbb{Z}(2^{m_2}) \times \cdots \times 
\mathbb{Z}(2^{m_k}), \quad m_1 < m_2 < \cdots < m_k,
\end{equation}
or $X = \mathbb{Z}(2^\infty)$, or $X = \Delta_2$.  
Let $\alpha_j, \beta_j$, $j=1,2,\dots$, be topological automorphisms 
of $X$, and let $\xi_j$ be independent random variables 
taking values in $X$  with distributions $\mu_j$.  
Assume that the following conditions hold:
\renewcommand{\labelenumi}{\rm(\roman{enumi})}
\begin{enumerate}

\item

$L_1=\sum\limits_{j=1}^\infty\alpha_j\xi_j$ and $L_2=\sum\limits_{j=1}^\infty\beta_j\xi_j$;

\item

the linear forms $L_1$ and $L_2$ are independent.
  
\end{enumerate}
Then all $\mu_j$ are degenerate distributions.
\end{theorem}
\begin{proof}
Denote by $Y$ the character group of the group $X$.

1. Assume that the group $X$ is of the form~(\ref{ne10}).
By Lemma~\ref{l1}, it follows from conditions~(i) and~(ii) that the characteristic functions
$\widehat\mu_j(y)$ satisfy equation~(\ref{e2}).
Set $\nu_j = \mu_j * \bar\mu_j$. Then $\widehat\nu_j(y) = |\widehat\mu_j(y)|^2 \ge 0$ for all $y \in Y$.
The characteristic functions $\widehat\nu_j(y)$ also satisfy equation~(\ref{e2}).
Denote by $\zeta_j$ the independent random variables taking values in the group $X$
with distributions $\nu_j$. We may assume, without loss of 
generality, that $\alpha_j = I$, $j = 1, 2, \dots$.
The series   
$\sum\limits_{j=1}^\infty\zeta_j$ and $\sum\limits_{j=1}^\infty\beta_j\zeta_j$
converge  in distribution.
Let 
\begin{equation}
\label{08.02.26.2}
M_1=\sum\limits_{j=1}^\infty\zeta_j\  \text{and} \  M_2=\sum\limits_{j=1}^\infty\beta_j\zeta_j.
\end{equation}
 By Lemma~\ref{l1}, the linear forms $M_1$ and $M_2$ are independent.
Since $X$ is a finite group, the group $\operatorname{Aut}(X)$ is also finite.
This implies that there are only finitely many distinct automorphisms $\beta_j$.
By renumbering the finite number of the random variables $\zeta_j$, we may assume
that these distinct automorphisms are $\beta_1, \beta_2, \dots, \beta_n$.
We now introduce a new numbering for the random variables $\zeta_j$. Namely, for
each fixed $k$, $k = 1, 2, \dots, n$, denote by $\zeta_{l,k}$, $l = 1, 2, \dots$, 
the random variables $\zeta_j$ which have coefficient $\beta_k$ in the series
$\sum_{j=1}^\infty \beta_j \zeta_j$. Without loss of generality and to 
simplify the notation, we assume that the number 
of such random variables $\zeta_{l,k}$ is infinite for each $k$. The case where 
this number is finite for some $k$ is treated analogously.
Every of the series 
 $\sum\limits_{l=1}^\infty\zeta_{l, k}$,
$k=1, 2, \dots, n$, converges in distribution.
Denote by $\eta_k$ its sum, i.e., 
\begin{equation*}
\eta_k=\sum_{l=1}^\infty\zeta_{l, k}, \quad k=1, 2,\dots, n.
\end{equation*}
Then 
$$M_1=\sum\limits_{k=1}^n\eta_k\  \text{and} \  
M_2=\sum\limits_{k=1}^n\beta_k\eta_k.$$
We have two independent linear forms $M_1$ and $M_2$ of a finite number of independent
random variables $\eta_k$ taking values in a group of the form~(\ref{ne10}). 
As shown in~\cite{Fe10}, it follows that all $\eta_k$ have degenerate distributions,
and consequently, all $\nu_j$ are degenerate distributions. This, in turn, implies that
all $\mu_j$ are degenerate distributions.

2. Let $X={\mathbb Z}(2^\infty)$.
As   $Y$ is topologically isomorphic to the group $\Delta_2$, 
to avoid introducing new notation, suppose $Y=\Delta_2$.   
  By Lemma~\ref{l1}, it follows from
conditions (i) and (ii) that the characteristic functions  $\widehat\mu_j(y)$ 
satisfy equation~(\ref{e2}) and 
the infinite products in the both sides of equation~(\ref{e2})
 converge uniformly on the group $\Delta_2^2$. 
 
Arguing as in the proof of Step~1, it suffices prove the proposition assuming that 
 $\widehat\mu_j(y)\ge 0$ for all $y \in \Delta_2$, $j = 1, 2, \dots$. 
The family of the subgroups
 $\{2^k\Delta_2\}_{k=0}^{\infty}$ forms an open 
 basis at zero of the group $\Delta_2$. It follows from this that 
 we can take a natural number $m$ in such a way that 
 $\widehat\mu_{L_i}(y)> 0$ for all
 $y\in 2^m\Delta_2$, $i=1, 2$. Each topological automorphism 
 of the group $\Delta_2$ is a multiplication
 by an invertible element of the ring $\Delta_2$
 (\!\!\cite[(26.18(e))]{HeRo1}). This implies that 
 all subgroups $2^k\Delta_2$ of the group $\Delta_2$ are
 characteristic. Consider the restriction of 
 equation~(\ref{e2}) to the subgroup $2^m\Delta_2$. 
 Setting  $f_j(y) = - \ln{\widehat \mu_j(y)}$, 
 $y\in 2^m\Delta_2$,  and  
 arguing as in the proof of Lemma~\ref{th1},
 we make sure that  $\widehat\mu_j(y)=1$  for all $y \in 2^m\Delta_2$, $j=1, 2, \dots$. 
 By Lemma~\ref{l4}, all distributions $\mu_j$ are supported in
 the annihilator $A({\mathbb Z}(2^\infty), 2^m\Delta_2)={\mathbb Z}(2^m)$, i.e.,  
 the random variables  $\xi_j$ take values in ${\mathbb Z}(2^m)$.
 The subgroup ${\mathbb Z}(2^m)$ of the group ${\mathbb Z}(2^\infty)$ 
 is characteristic. Thus the proof  
 reduces to the
 case when $X={\mathbb Z}(2^m)$, i.e., to Step~1.
  
 3. Let $X=\Delta_2$. 
Since $Y$ is topologically isomorphic to the group 
${\mathbb Z}(2^\infty)$, 
to avoid introducing new notation, suppose $Y={\mathbb Z}(2^\infty)$.
By Lemma~\ref{l1}, it follows from
conditions (i) and (ii) that the characteristic functions  $\widehat\mu_j(y)$ 
satisfy equation~(\ref{e2}).

Arguing as in the proof of Step~1, it suffices prove the proposition assuming that 
 $\widehat\mu_j(y)\ge 0$ for all $y \in {\mathbb Z}(2^\infty)$, $j = 1, 2, \dots$. 
 The subgroups ${\mathbb Z}(2^k)$, $k=0, 1, 2, \dots$, 
  of the group ${\mathbb Z}(2^\infty)$ are characteristic.
  Consider the restriction of equation~(\ref{e2}) to the subgroup
  ${\mathbb Z}(2^k)$. By Step~1, the statement of the proposition is true
  for the group ${\mathbb Z}(2^k)$. Taking into account that $\widehat\mu_j(y)\ge 0$ 
  and Lemma~\ref{l1}, this implies that 
  $\widehat\mu_j(y)=1$  for all $y \in {\mathbb Z}(2^k)$, $k=0, 1, 2, \dots$.
By Lemma~\ref{l4}, each of the distributions $\mu_j$ is supported in
the annihilators $A(\Delta_2, {\mathbb Z}(2^k))=2^k\Delta_2$,  $k=0, 1, 2, \dots$.
Since $\bigcap\limits_{k=0}^\infty 2^k\Delta_2=\{0\}$, we obtain that
$\mu_j=E_0$, $j=1, 2, \dots$.
\end{proof}

In \cite{R} B.~Ramachandran proved the following theorem.
\begin{R}{\rm(\!\!\cite[Theorem 8.2.2]{R})}
Let $\alpha_j$, $\beta_j$, $j=1, 2, \dots$, be nonzero real numbers, 
and let $\xi_j$ be 
real-valued independent random variables  with distributions $\mu_j$. 
Assume that the following conditions hold:
\renewcommand{\labelenumi}{\rm(\roman{enumi})}
\begin{enumerate}
\item	
$L_1=\sum\limits_{j=1}^\infty \alpha_j\xi_j$ and 
$L_2=\sum\limits_{j=1}^\infty \beta_j\xi_j$;
\item
the linear forms $L_1$ and $L_2$ are independent;
\item	
 the sequences $\{\alpha_j\beta_j^{-1}\}$   and $\{\beta_j\alpha_j^{-1}\}$
are bounded.
\end{enumerate}
Then all $\mu_j$ are Gaussian distributions.
\end{R}

In the final part of the paper, we prove some group analogues of this theorem. 
We recall the definition of the Gaussian distribution on a second-countable 
locally compact Abelian group $X$. Let $Y$ be the character group of $X$. 

A distribution $\gamma$ on $X$ is called Gaussian 
if its characteristic function can be represented as
	\begin{equation*}\label{fe22.3}\widehat\gamma(y)=
	(x,y)\exp\{-\varphi(y)\}, \quad y \in Y,
\end{equation*}
where $x\in X$  and $\varphi(y)$ is a continuous nonnegative	
function on the group $Y$ satisfying the equation
\begin{equation*}\label{fe22.4} 
	\varphi(u+v)+\varphi(u-v)=2[\varphi(u)+\varphi(v)], \quad  u, v
\in Y.
\end{equation*}

We note that a distribution $\gamma$ on an 
${\bm a}$-adic solenoid $\Sigma_{\bm a}$ is Gaussian if its characteristic 
function is represented in the form 
$$
\widehat\gamma(y)=(x, y)\exp\{-\sigma y^2\}, \quad y\in H_{\bm a},
$$
where $x\in \Sigma_{\bm a}$, $\sigma\ge 0$. 

The first lemma is a group analogue of the
Skitovich--Darmois theorem.

\begin{lemma} [\!\!\protect\cite{Fe16}]\label{s10.2} 
 Let $X$ be a second-countable locally compact Abelian group that 
 contains no subgroups topologically isomorphic to the circle group 
 $\mathbb{T}$. Let  $\alpha_j$,  $\beta_j$, 
$j = 1, 2,\dots, n$,  $n \ge 2$, be topological automorphisms 
of $X$. Let
$\xi_j$ be independent random variables taking values in $X$  with
distributions  $\mu_j$ having nonvanishing characteristic
functions.  If the linear forms $L_1 =
\alpha_1\xi_1 + \dots + \alpha_n\xi_n$ and $L_2 
= \beta_1\xi_1 +
\dots + \beta_n\xi_n$ are independent, then   
 all  $\mu_j$  are Gaussian distributions.
\end{lemma}

The second lemma is a group analogue of Cram\'er's decomposition theorem for 
the Gaussian distribution on the real line.

\begin{lemma}[{\!\!\protect\cite[Theorem 4.6]{FeBook}}]
 \label{ns10.2} 
 Let $X$ be a second-countable locally compact Abelian group that 
 contains no subgroups topologically isomorphic to the circle group 
 $\mathbb{T}$. Let  $\mu$ be a Gaussian distribution
on $X$, and let $\mu=\mu_1*\mu_2$, where $\mu_1$ and $\mu_2$ 
are probability distributions.
Then  $\mu_j$ are Gaussian distributions.
\end{lemma}

For a locally compact Abelian group $X$ and a nonzero integer $n$, 
denote by $f_n \colon X \to X$ 
the endomorphism of $X$ defined by
$f_n(x) = n x$, $x \in X$. 
Consider an ${\bm a}$-adic solenoid $\Sigma_{\bm a}$.
 Each topological automorphism $\alpha\in {\rm
Aut}(\Sigma_{\bm a})$
 is of the form
$\alpha = f_m f_n^{-1}$ for some mutually prime $m$ and $n$,
where $f_m, f_n \in {\rm
Aut}(\Sigma_{\bm a})$. We identify $\alpha
=  f_m f_n^{-1}$ with the rational number $\frac{m}{n}$. 

The following statement can be regarded as an analogue of Ramachandran’s theorem for 
${\bm a}$-adic solenoids.

\begin{proposition}\label{np1}
Let $\Sigma_{\bm a}$ be an ${\bm a}$-adic solenoid such that  there exists a
unique prime number $p$ such that $\Sigma_{\bm a}$ contains no elements of order $p$.
Let $\alpha_j, \beta_j$, $j = 1, 2, \dots$, be topological automorphisms of
$\Sigma_{\bm a}$, and let $\xi_j$ be independent random variables taking values in
$\Sigma_{\bm a}$  with distributions $\mu_j$.
Assume that the following conditions hold:
\renewcommand{\labelenumi}{\rm(\roman{enumi})}
\begin{enumerate}

\item

$L_1=\sum\limits_{j=1}^\infty\alpha_j\xi_j$ and $L_2=\sum\limits_{j=1}^\infty\beta_j\xi_j$;

\item

the linear forms $L_1$ and $L_2$ are independent;

\item	

the characteristic functions $\widehat\mu_{L_1}(y)$ and $\widehat\mu_{L_2}(y)$ do not vanish;

\item	
 the sequences $\{\alpha_j\beta_j^{-1}\}$   and $\{\alpha_j^{-1}\beta_j\}$
are bounded.
  
\end{enumerate}
Then all $\mu_j$ are Gaussian distributions.
\end{proposition}

\begin{proof}
Note that, although this fact will not be used in the proof, there exists 
a unique prime number $p$ such that an ${\bm a}$-adic solenoid $\Sigma_{\bm a}$ 
contains no elements of order $p$ if and only if there exists a unique prime number 
$p$ that divides infinitely many of the numbers $a_j$.

We may assume, without loss of generality, that $\alpha_j = I$ for all
$j = 1, 2, \dots$.
Then condition~(iv) is transformed into the following condition:
\begin{equation}
\label{08.02.26.1}
\text{the sequences}\ \{\beta_j^{-1}\}\ \text{and}\ \{\beta_j\}
\ \text{are bounded}.
\end{equation}

Since $\Sigma_{\bm a}$ is a connected compact Abelian group, the mapping $f_n$ is an
epimorphism for each nonzero integer $n$. As $\Sigma_{\bm a}$ contains
 no elements of order
$p$, we conclude that $f_p$ is a monomorphism. Taking into account that $f_p$ is also an
epimorphism, we obtain $f_p \in \operatorname{Aut}(\Sigma_{\bm a})$.
This implies that $f_p\in \operatorname{Aut}(H_{\bm a})$. It means that both 
$\Sigma_{\bm a}$ and $H_{\bm a}$ are
groups with unique division by $p$. 
Since $p$ is the unique prime number such that $\Sigma_{\bm a}$ contains no elements
of order $p$, it follows that if $q$ is a prime number with $q \ne p$, then
$f_q \notin \operatorname{Aut}(\Sigma_{\bm a})$. 
Taking into account the form of a topological automorphisms of the group
$\Sigma_{\bm a}$, we conclude that
 every topological automorphism
of $\Sigma_{\bm a}$ is of the form $\pm f_{p^m}^{\pm 1}$, where $m$ is a
nonnegative integer. It follows from this that there exist only
finitely many distinct topological 
automorphisms $\beta_j$ satisfying condition~(\ref{08.02.26.1}).

We proceed with the argument in the same way as in the proof of Step~1 of
Theorem~\ref{th5}, retaining the notation used there. It follows from the above
that the series in~(\ref{08.02.26.2}) contains only finitely many distinct 
topological automorphisms
$\beta_j$.
Thus, we obtain two independent linear forms $M_1$ and $M_2$ of a finite number of
independent random variables $\eta_k$ taking values in the ${\bm a}$-adic 
solenoid $\Sigma_{\bm a}$. Condition (iii) of the proposition implies
that all the characteristic functions $\widehat\mu_{\eta_k}(y)$ do not vanish. 
Since any ${\bm a}$-adic solenoid is a second-countable locally compact Abelian group
containing no subgroups topologically isomorphic to the circle group $\mathbb{T}$,
it follows from Lemma~\ref{s10.2} that all distributions
$\mu_{\eta_k}$ are Gaussian. By Lemma~\ref{ns10.2}, this implies that 
all distributions $\nu_j$, and
hence all distributions $\mu_j$  are also Gaussian.
\end{proof}

The proof of Proposition~\ref{np1} is based on Lemmas~\ref{s10.2} and~\ref{ns10.2}, 
which hold for any second-countable locally compact Abelian group that contains 
no subgroups topologically isomorphic to the circle group $\mathbb{T}$. 
Furthermore, condition~(iv) of Proposition~\ref{np1} was used only to ensure that 
the set of distinct topological automorphisms appearing in the linear forms under 
consideration is finite.
Taking this into account, the following statement holds.
 
 \begin{proposition}\label{nnp1}
Let $X$ be a second-countable locally compact Abelian group that contains 
no subgroups topologically isomorphic to the circle group $\mathbb{T}$ and 
has a finite topological automorphism group $\operatorname{Aut}(X)$.
Let $\alpha_j, \beta_j$, $j = 1, 2, \dots$, be topological automorphisms of
$X$, and let $\xi_j$ be independent random variables taking values in
$X$  with distributions $\mu_j$.
Assume that the following conditions hold:
\renewcommand{\labelenumi}{\rm(\roman{enumi})}
\begin{enumerate}

\item

$L_1=\sum\limits_{j=1}^\infty\alpha_j\xi_j$ and $L_2=\sum\limits_{j=1}^\infty\beta_j\xi_j$;

\item

the linear forms $L_1$ and $L_2$ are independent;

\item	

the characteristic functions $\widehat\mu_{L_1}(y)$ and $\widehat\mu_{L_2}(y)$ do not vanish.
  
\end{enumerate}
Then all $\mu_j$ are Gaussian distributions.
\end{proposition}

\noindent\textbf{Acknowledgements} I would like to express my sincere gratitude 
to the referee for a careful and thorough reading of the manuscript. The referee's 
valuable comments and suggestions have helped me to significantly improve 
the presentation of the results in the article.

\bigskip

\noindent{\Large\bfseries Declarations}

\bigskip

\noindent\textbf{Funding} The author has not received any funding for this research.

\bigskip

\noindent\textbf{Conflict of interests} The author states that 
there is no conflict of interest.

\bigskip

\noindent\textbf{Ethical Approval} Not applicable.

\bigskip

\noindent\textbf{Data Availability Statement} Data sharing is not applicable to this article 
as no datasets were generated or analysed during the current study.

\medskip

\noindent B. Verkin Institute for Low Temperature Physics and Engineering\\
of the National Academy of Sciences of Ukraine\\
47, Nauky ave, Kharkiv, 61103, Ukraine

\medskip

\noindent e-mail:    gennadiy\_f@yahoo.co.uk

\end{document}